\documentclass[hidelinks,onefignum,onetabnum]{siamart251216}
\usepackage{amssymb,amsmath,mathtools,amsfonts}
\usepackage{graphicx,subfig, version, color}
\usepackage{booktabs}
\usepackage{bm}
\usepackage{color}
\usepackage{xcolor}
\usepackage{placeins}
\numberwithin{equation}{section}
\numberwithin{figure}{section}
\numberwithin{table}{section}
\usepackage{framed,multirow}
\newcommand{\R}{\mathbb{R}}
\newcommand{\N}{\mathbb{N}}
\newcommand{\Bx}{\bm{x}}
\newcommand{\Bf}{\bm{f}}

\newcommand{\Bq}{\bm{q}}
\newcommand{\Br}{\bm{r}}
\newcommand{\Bv}{\bm{v}}
\newcommand{\Bu}{\bm{u}}
\newcommand{\Bd}{\bm{d}}
\newcommand{\BA}{\bm{A}}
\newcommand{\BB}{\bm{B}}

\newcommand{\BC}{\bm{C}}

\newcommand{\Be}{\bm{e}}
\newcommand{\By}{\bm{y}}
\newcommand{\Bz}{\bm{z}}

\newcommand{\Bb}{\bm{b}}
\newcommand{\Bc}{\bm{c}}

\newcommand{\BW}{\bm{W}}

\newcommand{\Btheta}{\bm{\theta}}

\newcommand{\calA}{\mathcal{A}}
\newcommand{\calT}{\mathcal{T}}
\newcommand{\calN}{\mathcal{N}}

\newcommand{\calI}{\mathcal{I}}
\newcommand{\hf}{\hat{f}}

\newcommand{\odt}{\mathrm{d}t}

\DeclareMathOperator{\Ker}{Ker}

\newsiamremark{remark}{Remark}
\newsiamremark{hypothesis}{Hypothesis}
\crefname{hypothesis}{Hypothesis}{Hypotheses}
\crefname{equation}{equation}{equations}

\makeatletter
\g@addto@macro\normalsize{%
  \setlength{\abovedisplayskip}{6pt plus 2pt minus 2pt}%
  \setlength{\belowdisplayskip}{7pt plus 2pt minus 2pt}%
  \setlength{\abovedisplayshortskip}{4pt plus 1pt minus 1pt}%
  \setlength{\belowdisplayshortskip}{6pt plus 2pt minus 2pt}%
}
\makeatother

\headers{Kernel Localization and Whole-Trajectory Generalization}{Y. Liu and Y. Gu}

\title{Kernel Localization and Whole-Trajectory Generalization for Linear Multistep Methods in Deep Learning-Based Discovery of Dynamical Systems\thanks{Submitted to the editors DATE.
\funding{This work was funded by NSFC Grant 92370101.}}}

\author{Yaru Liu\thanks{School of Mathematical Sciences, University of Electronic Science and Technology of China, Sichuan 611731, China (\email{yaruliu@std.uestc.edu.cn},\email{yiqigu@uestc.edu.cn}).}
\and Yiqi Gu\footnotemark[2]}

\ifpdf
\hypersetup{
 pdftitle={Kernel Localization and Whole-Trajectory Generalization for Linear Multistep Methods in Deep Learning-Based Discovery of Dynamical Systems},
 pdfauthor={Y. Liu and Y. Gu}
}
\fi

\begin{document}
\sloppy
\maketitle
\begin{abstract}
Linear multistep methods (LMMs) combined with neural-network approximation provide a high-order framework for learning governing vector fields of dynamical systems from discrete trajectory data. This paper studies two issues in LMM-based discovery that are not resolved by existing grid-level convergence theory. First, in non-auxiliary Adams--Bashforth (A-B) and Adams--Moulton (A-M) discovery systems, we observe that zero-residual grid solutions are nonunique but consistent along the trajectory, with differences limited to the boundary layer. We explain this phenomenon through a kernel analysis of the non-auxiliary discovery matrices. Under the corresponding discovery-stability conditions, the differences between zero-residual grid solutions are exponentially localized near the initial indices for A-B schemes, whereas they form two-sided boundary layers near the initial and terminal indices for A-M schemes. Second, we derive whole-trajectory generalization estimates for both auxiliary and non-auxiliary formulations. Once the learned vector field is restricted to a fixed observed trajectory, each component of the error becomes a scalar function of time. For the auxiliary formulation, the trajectory error is $O(h^p)$ under the corresponding grid accuracy, approximation, and trace regularity assumptions. For non-auxiliary formulations, the global estimates contain additional boundary-layer terms. On fixed interior subintervals, these terms are exponentially damped. Numerical experiments illustrate the convergence behavior.
\end{abstract}



\section{Introduction}\label{sec:introduction}
Recovering a governing vector field from discrete trajectory data is a basic inverse problem in numerical analysis and scientific computing. Let $d\geq 1$ be the dimension of the state space, and suppose that the observed states are generated by the autonomous system
\begin{equation}\label{eq:intro_ode}
\begin{cases}
\dot{\Bx}(t)=\Bf(\Bx(t)), \qquad 0<t\leq T,\\
\Bx(0)=\Bx_{\rm init},
\end{cases}
\end{equation}
where $\Bx:[0,T]\to\R^d$ is a sufficiently smooth state trajectory, $\Bf:\R^d\to\R^d$ is an unknown vector field, and $\Bx_{\rm init}\in\R^d$ is the initial state. Let $t_n=nh$, $h=T/N$, for $0\leq n\leq N$, and let $\Bx_n=\Bx(t_n)$ denote the corresponding equispaced trajectory samples. The discovery problem is to construct a data-driven approximation of $\Bf$, at least along the observed trajectory. Such inverse problems arise in reduced-order modeling, forecasting, control, and the scientific interpretation of physical and biological systems. From an analytical viewpoint, the problem combines data-driven learning and numerical approximation. The observations are constrained by a differential equation, while the unknown object is a vector-valued function. A convergence theory must account for time discretization, approximation error and algebraic stability.

There is a broad literature on data-driven discovery of dynamical systems. Existing approaches include symbolic regression \cite{Bongard2007,Schmidt2009}, sparse identification methods  \cite{Brunton2016,Rudy2017,Zhang2018}, Gaussian-process and other probabilistic or statistical methods \cite{Kocijan2005,Raissi2017_2,Lu2019}, orthogonal-polynomial and regression-based approaches \cite{Qin2019,Wu2019}, and neural-network-based formulations \cite{Raissi2019,Chen2018,Long2019,Sun2020}. These methods differ in the prior structure imposed on the vector field and in the way observational data are converted into constraints on the unknown dynamics. For trajectory data, the temporal ordering of the samples provides a natural source of additional structure. Instead of estimating derivatives pointwise, one may impose a time-discretization residual on the observed trajectory. This point of view leads to linear multistep methods (LMMs) for discovery, in which the unknown vector field is constrained through a discrete residual equation. The focus of this paper is not the design of a new network architecture or optimization algorithm, but the finite-dimensional algebraic system induced by this LMM residual. In particular, we study how its kernel structure, stability properties, and truncation errors determine the accuracy of the recovered vector field on the grid and along the whole observed trajectory.

To make this discrete algebraic system precise, recall that LMMs are classical high-order time discretizations for initial value problems \cite{Dahlquist1956,Dahlquist1963,Henrici1962,Hairer1993}. In the forward problem, the vector field is known and the multistep formula is used to compute the unknown states. In the inverse setting considered here, the same formula is used in the opposite direction: the observed states are inserted into the scheme, while the values of the unknown vector field at these states are represented by an approximation from a prescribed class. Let $\{\alpha_m\}_{m=0}^M$ and $\{\beta_m\}_{m=0}^M$ be the real coefficients of an $M$-step LMM, with $\alpha_0\neq 0$. Applied to the sampled trajectory, the method gives
\begin{equation}\label{eq:intro_lmm}
\sum_{m=0}^M \alpha_m \Bx_{n-m}=h\sum_{m=0}^M \beta_m \Bf(\Bx_{n-m}),\qquad M\leq n\leq N,
\end{equation}
up to the local truncation residual. We can replace $\Bf$ by a neural network, and more generally, by an element of an admissible approximation class, leading to a residual minimization problem for the governing vector field. This LMM-based viewpoint has been used in several computational works
\cite{Raissi2018_4,Rudy2019,Tipireddy2019,Xie2019}, and its numerical analysis foundations have been developed through stability and convergence studies \cite{Keller2021,Du2022}.

Within the LMM framework, we focus on two classical families: the Adams--Bashforth (A-B) schemes and the Adams--Moulton (A-M) schemes. After the state data are substituted into \eqref{eq:intro_lmm}, the unknowns are the grid values of one component of the vector field at the involved sample states. But, for A-B and A-M schemes, the number of these unknown grid values is larger than the number of LMM equations. Thus, without additional conditions, the corresponding discrete inverse system is underdetermined. To remove this nonuniqueness, the convergence theory in \cite{Du2022} augments the A-B and A-M systems by imposing auxiliary conditions, such as finite difference approximations of selected initial values of the vector field. The resulting augmented system is square and nonsingular, and the values of the learned vector field at the involved grid points are uniquely determined at the discrete level. For one scalar component $f$ of the vector field and an admissible approximation class $\calA$, this theory gives a grid estimate of the form
\begin{equation}\label{eq:intro_grid_estimate}
|\hf_{\calA,h}-f|_{2,h}\leq C\kappa_2(\BA_h)\left(h^p+e_{\calA}\right),
\end{equation}
where $p$ is the order of the LMM, $\hf_{\calA,h}$ is its approximation computed by the methods, $e_{\calA}$ is the approximation error of $\calA$, $\BA_h$ is the augmented discovery matrix, and $\left|\cdot\right|_{2,h}$ denotes the discrete seminorm over the involved sample states. In particular, when $\kappa_2(\BA_h)$ remains uniformly bounded as $h\to 0$, the estimate explains the observed grid accuracy of stable LMM-based discovery methods.

The auxiliary conditions are useful for the stability analysis, but they modify the algebraic system to be solved. In particular, they select a unique grid vector by adding information that is not contained in the original LMM residual. If these auxiliary conditions are not imposed, the A-B and A-M discovery systems have infinitely many grid solutions. However, we observe (see \Cref{fig:A-B--A-M}) that their behavior is more ordered than what simple nonuniqueness of the solution would imply. For A-B schemes, the discrepancy in the recovered grid vector is usually localized near the first few time levels, while the later values agree well with the target vector field. For A-M schemes, a similar phenomenon occurs, but the discrepancy may appear near both the initial and terminal indices. These observations suggest that the kernel of the corresponding underdetermined discovery matrix has a boundary layer structure.
\begin{figure}[htbp]
\centering
\includegraphics[width=0.89\textwidth]{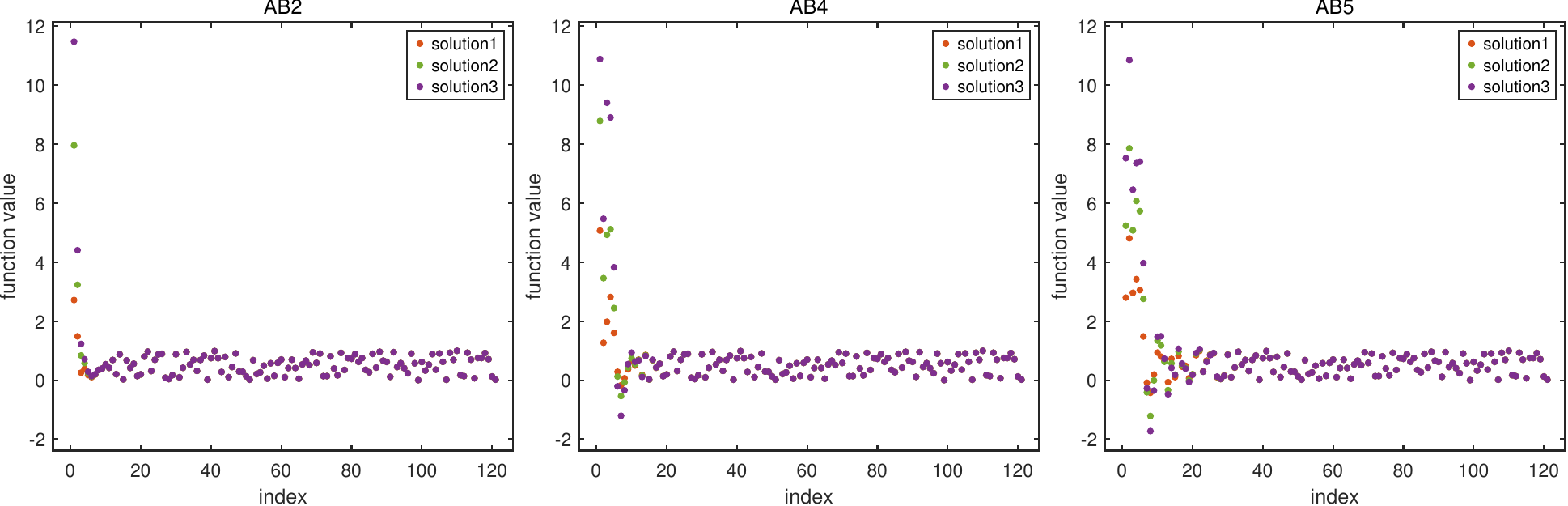}\\[6pt]
\includegraphics[width=0.89\textwidth]{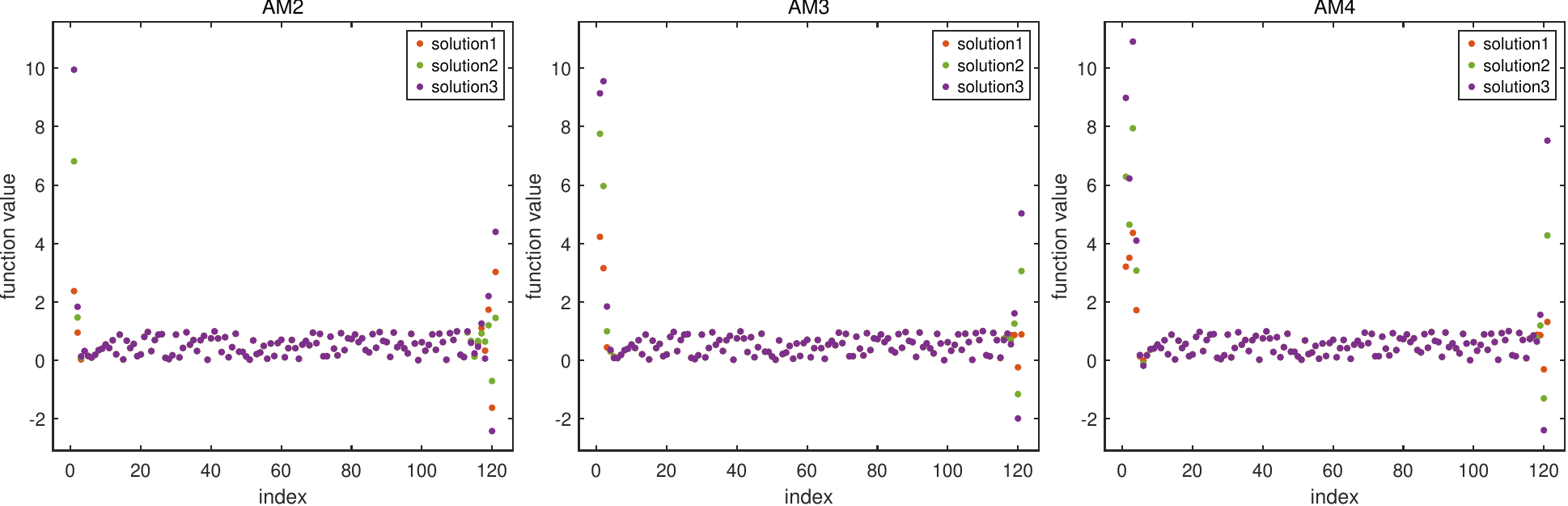}
\caption{Boundary-layer localization of nonuniqueness in non-auxiliary A-B and A-M discovery systems. To illustrate the algebraic kernel structure, we use the scalar time-dependent relation $\dot x=f(t)$ on $[0,1]$, with randomly sampled values of $f$ on the uniform grid. This simplification affects only the presentation of the example, since the kernel localization depends on the LMM discovery matrix. For each scheme, three representative zero-residual solutions are plotted.}
\label{fig:A-B--A-M}
\end{figure}

These observations motivate the first main result of the paper: an analysis of the kernels of the non-auxiliary A-B and A-M discovery matrices. In \cite{Du2022}, the authors proposed that the favorable empirical behavior of the non-auxiliary neural-network formulation is related to the implicit regularization during the training process. Our analysis shows that this phenomenon is closely related to the special algebraic structure of the A-B and A-M discovery matrices. More precisely, the nonuniqueness of the non-auxiliary systems is determined by the boundary-layer structure of the kernel vectors. Therefore, two zero-residual grid solutions to the same residual equation may look almost identical away from the boundary indices. For A-B schemes satisfying the discovery-stability condition associated with the A-B discovery polynomial, every kernel vector is exponentially localized near the initial indices. In a representative form, if $\Bd=(d_0,\ldots,d_{N-1})^\top$ belongs to the kernel of the non-auxiliary $M$-step A-B discovery matrix ($M\geq 2$), there exist constants $C>0$ and $\theta\in(0,1)$, independent of $N$ and $h$, such that
\begin{equation*}
|d_i|\leq C\theta^i\sum_{j=0}^{M-2}|d_j|,\qquad 0\leq i\leq N-1.
\end{equation*}
Thus, the difference between two non-auxiliary A-B numerical solutions is determined by their first few components and decays exponentially within the time interval. For A-M schemes satisfying the corresponding discovery-stability condition, the corresponding estimate has a two-sided form
\begin{equation*}
|d_i|\leq C\left(\theta^i E_0+\theta^{N-i}E_T\right),\qquad 0\leq i\leq N,
\end{equation*}
where $E_0=\sum_{j=0}^{M-1}|d_j|$ and $E_T=\sum_{j=N-M+1}^{N}|d_j|$ measure the left and right endpoint data, respectively.

The second contribution is to lift grid-level estimates to continuous-time estimates along the fixed observed trajectory. The estimate \eqref{eq:intro_grid_estimate} controls the error only at the sampled states, whereas the learned vector field is evaluated along the continuous trajectory. Once the trajectory $\Bx(t)$ is fixed, the error
\begin{equation*}
g_h(t)=\hf_{\calA,h}(\Bx(t))-f(\Bx(t))
\end{equation*}
is a scalar function of the single variable $t\in[0,T]$. Thus, the transformation from grid error to trajectory error is essentially a one-dimensional sampling problem on the observed trajectory, rather than a high-dimensional statistical learning problem over the full state space. We use deterministic sampling inequalities to control trajectory norms of $g_h$ by its grid values and Sobolev seminorms of the trace error. For the auxiliary formulation, this gives
\begin{equation*}
\int_0^T |\hf_{\calA,h}(\Bx(t))-f(\Bx(t))|~\odt
\leq C\Big(\kappa_2(\BA_h)(h^p+e_{\calA})+h^p\tilde{R}\Big).
\end{equation*}
Consequently, if $\kappa_2(\BA_h)$ is uniformly bounded, $e_{\calA}=O(h^p)$, and $\tilde{R}=O(1)$, then the whole-trajectory error is also of order $O(h^p)$. For the non-auxiliary formulation, combining the sampling estimates with the boundary-layer analysis yields trajectory error bounds with additional endpoint-layer contributions. In the A-B case, the global estimate contains a left boundary-layer term of size $O(hE_0)$, while in the A-M case it contains two endpoint-layer terms of size $O(h(E_0+E_T))$. These terms may limit the global order when the boundary-layer errors are merely bounded. However, on fixed interior subintervals, they are multiplied by exponentially small factors of the form $\theta^{\delta/h}$ and can be absorbed into the $O(h^p)$ terms. Thus, the interior trajectory error retains the high-order behavior of the LMM.

The remainder of the paper is organized as follows. In \Cref{sec:prelim}, we introduce the LMM discovery system and the notation used throughout the paper. In \Cref{sec:kernel}, we study the kernel structure of the non-auxiliary A-B and A-M discovery matrices. In \Cref{sec:generalization}, we establish whole-trajectory generalization estimates for the auxiliary and non-auxiliary A-B and A-M schemes. In \Cref{sec:numerics}, numerical experiments are presented to illustrate the theoretical results. Finally, \Cref{sec:conclusion} concludes the paper.

\section{Preliminaries and notation}\label{sec:prelim}
In this section, we fix the notation used in the rest of the paper. We recall the LMM formulation of dynamics discovery on a single trajectory and distinguish the non-auxiliary and auxiliary algebraic systems. Since the components of the vector field enter the LMM residual independently, the analysis below is written for one scalar component of the governing vector field.

\subsection{Trajectory data}
We consider the system \eqref{eq:intro_ode}. The observed trajectory is denoted by
\begin{equation*}
\calT=\{\Bx(t):0\leq t\leq T\}.
\end{equation*}
Let $N\in\N$, $h=T/N$, and $t_n=nh$ for $0\leq n\leq N$. The available data are the equidistant samples
\begin{equation*}
\Bx_n=\Bx(t_n),\qquad n=0,1,\ldots,N.
\end{equation*}
For a scalar component $f$ of $\Bf$, we write $f_n$ for the discrete unknown intended to approximate $f(\Bx_n)$. Since the LMM residual is imposed componentwise, all vector-valued estimates can be obtained by applying the scalar analysis to each component of $\Bf$.

\subsection{The LMM discovery system}
Let $\{\alpha_m\}_{m=0}^M$ and $\{\beta_m\}_{m=0}^M$ be the coefficients of an $M$-step LMM. Substituting the observed states into the LMM formula gives the scalar discovery equations
\begin{equation}\label{eq:lmm_discrete}
h\sum_{m=0}^M \beta_m f_{n-m}=\sum_{m=0}^M \alpha_m x_{n-m},\qquad n=M,\ldots,N,
\end{equation}
where $x_n$ denotes the state component corresponding to the scalar vector-field component $f$. 

As in \cite{Du2022}, let $s$ and $e(N)$ be the first and last indices for which $f_s,\ldots,f_{e(N)}$ appear in \eqref{eq:lmm_discrete}. We denote the corresponding set of involved indices by
\begin{equation*}
\calI_h=\{s,s+1,\ldots,e(N)\},
\end{equation*}
and set
\begin{equation*}
t(N)=|\calI_h|=e(N)-s+1.
\end{equation*}
For an $M$-step A-B scheme, one has
\begin{equation*}
s=0,\qquad e(N)=N-1,\qquad t(N)=N,
\end{equation*}
whereas for an $M$-step A-M scheme,
\begin{equation*}
s=0,\qquad e(N)=N,\qquad t(N)=N+1.
\end{equation*}
Define the vector of involved grid values by
\begin{equation*}
\Bf_h=\begin{bmatrix}f_s&f_{s+1}&\cdots&f_{e(N)}\end{bmatrix}^\top\in\R^{t(N)}.
\end{equation*}
The right-hand side vector is defined by
\begin{equation*}
\Bq_h=h^{-1}\left[\sum_{m=0}^M \alpha_m x_{M-m}\quad\sum_{m=0}^M \alpha_m x_{M+1-m}\quad \cdots \quad \sum_{m=0}^M \alpha_m x_{N-m}\right]^\top\in\R^{N-M+1}.
\end{equation*}
Then \eqref{eq:lmm_discrete} can be written as
\begin{equation}\label{eq:Bh_system}
\BB_h\Bf_h=\Bq_h,
\end{equation}
where 
\begin{equation*}
\BB_h=
\begin{bmatrix}
\beta_{M-s}&\beta_{M-s-1}&\cdots&\beta_{N-e(N)}&&&\\
&\beta_{M-s}&\beta_{M-s-1}&\cdots&\beta_{N-e(N)}&&\\
&&\ddots&\ddots&\ddots&\ddots&\\
&&&\beta_{M-s}&\beta_{M-s-1}&\cdots&\beta_{N-e(N)}
\end{bmatrix}
\end{equation*}
$\in\R^{(N-M+1)\times t(N)}$ is the non-auxiliary discovery matrix. The number of equations is $N-M+1$, while the number of unknown grid values is $t(N)$. For A-B schemes, the difference $t(N)-(N-M+1)=M-1$, so the system is underdetermined when $M\geq2$ and determined when $M=1$. For A-M schemes, the difference is $t(N)-(N-M+1)=M$, so the non-auxiliary A-M system is underdetermined for $M\geq1$.

\subsection{Non-auxiliary and auxiliary formulations}
We introduce the concept of fully connected neural networks (FNNs) and present network-based LMMs for non-auxiliary and auxiliary formulations for the discovery of dynamics from a trajectory. Let $\Omega_{\calT}\subset\R^d$ be a fixed neighborhood of $\calT$. Then FNN with depth $L$ and width $W$ is a parametrized function $u_{\Btheta}:\Omega_{\calT}\to\R$ of the form
\begin{equation*}
u_{\Btheta}(\Bz)=\BW_L\sigma(\BW_{L-1}\sigma(\dots\sigma(\BW_1\Bz+\Bb_1)\dots)+\Bb_{L-1})+\Bb_L, \quad \Bz\in\Omega_{\calT},
\end{equation*}
where $\BW_1\in\R^{W\times d}$, $\BW_2,\ldots,\BW_{L-1}\in\R^{W\times W}$, $\BW_{L}\in\R^{1\times W}$ are weights; $\Bb_1,\ldots,\Bb_{L-1}\in\R^{W\times 1}$, $\Bb_L\in\R$ are biases; $\sigma$ is some activation function which is applied entry-wise to a vector to obtain another vector of the same size; $\Btheta=\{\BW_\ell,\Bb_\ell\}_{\ell=1}^L$ is the set of all free parameters. 
Let $\Theta$ be the corresponding parameter space. We denote the corresponding network class by $\calN=\{u_{\Btheta}:\Omega_{\calT}\to\R, ~\Btheta\in\Theta\}$.

In the LMM formulation, a candidate function enters the residual only through its values at the sampled states $\Bx_s,\ldots,\Bx_{e(N)}$. Thus, the algebraic structure of the discrete problem depends only on the finite vector of grid values, not on the particular parametrization of the candidate function. Therefore, we formulate the non-auxiliary and auxiliary problems for a general admissible class $\calA\subset C(\Omega_{\calT})$. The neural-network case corresponds to $\calA=\calN$.

For $u\in\calA$, define its vector of involved grid values by
\begin{equation*}
\Bu=\begin{bmatrix}u(\Bx_s)&u(\Bx_{s+1})&\cdots&u(\Bx_{e(N)})\end{bmatrix}^\top\in\R^{t(N)}.
\end{equation*}
The non-auxiliary residual loss is obtained by replacing $\Bf_h$ in \eqref{eq:Bh_system} with $\Bu$:
\begin{equation*}
J_h(u)=\frac{1}{N-M+1}\|\BB_h\Bu-\Bq_h\|_2^2.
\end{equation*}
Equivalently,
\begin{equation}\label{eq:Jh_def}
J_h(u)=\frac{1}{N-M+1}\sum_{n=M}^N\left|\sum_{m=0}^M\beta_m u(\Bx_{n-m})-\sum_{m=0}^M h^{-1}\alpha_m x_{n-m}\right|^2.
\end{equation}
A non-auxiliary minimizer over $\calA$ is any element of $\operatorname*{argmin}_{u\in\calA}J_h(u)$. When $\BB_h$ has a nontrivial kernel, different grid-value vectors may produce the same zero LMM residual. This discrete nonuniqueness is the source of the boundary-layer phenomena studied below. The kernel analysis focuses on the grid-level nonuniqueness of zero-residual solutions. When the residual is not zero, the same algebraic estimates enter through the forced residual terms.

To remove the nonuniqueness of the involved grid values, we introduce auxiliary conditions and build an augmented loss function based on \eqref{eq:Jh_def}. Define $N_a=t(N)-(N-M+1)$. Let $\hat\BB_h\in\R^{N_a\times t(N)}$ and $\Bc_h\in\R^{N_a}$ encode the selected auxiliary equations. For the one-sided finite-difference auxiliary conditions used in \cite{Du2022}, $\hat\BB_h$ selects the first $N_a$ involved grid values and $\Bc_h$ contains the corresponding finite-difference approximations. More explicitly, for suitable finite-difference coefficients $\gamma_0,\ldots,\gamma_p$, the auxiliary equations are
\begin{equation}\label{eq:fd_aux_detailed}
u(\Bx_{s+r})=h^{-1}\sum_{m=0}^p\gamma_m x_{s+r+m},\qquad r=0,\ldots,N_a-1.
\end{equation}
The coefficients are chosen so that \eqref{eq:fd_aux_detailed} has local error $O(h^p)$ when evaluated on the exact trajectory. Set
\begin{equation*}
\BA_h=\begin{bmatrix}\hat\BB_h\\ \BB_h\end{bmatrix},\qquad \Bd_h=\begin{bmatrix}\Bc_h\\ \Bq_h\end{bmatrix}.
\end{equation*}
The auxiliary loss is defined by
\begin{equation*}
J_{a,h}(u)=\frac{1}{t(N)}\|\BA_h\Bu-\Bd_h\|_2^2.
\end{equation*}
For an admissible class $\calA$, let $\hf_{\calA,h}$ be a global minimizer of $J_{a,h}$. The grid estimate proved in \cite{Du2022} states that, under the smoothness and approximation assumptions described there,
\begin{equation*}
|\hf_{\calA,h}-f|_{2,h}\leq C\kappa_2(\BA_h)(h^p+e_{\calA}),
\end{equation*}
where $p$ is the order of the LMM and $e_{\calA}$ satisfies $e_{\calA}>\inf_{u\in\calA}\sup_{\Bz\in\Omega_{\calT}}|u(\Bz)-f(\Bz)|$; namely, $e_{\calA}$ is the approximation error bound between $\calA$ and $f$. Here, $|\cdot|_{2,h}$ denotes the empirical grid seminorm over the
sample states involved in the LMM, i.e., $|g|_{2,h}=\left({t(N)}^{-1}\sum_{n=s}^{e(N)}|g(\Bx_n)|^2\right)^{\frac 12}$ for all $g\in C(\calT)$. We will generalize this grid estimate from the sampled points to the whole continuous trajectory in \Cref{sec:generalization}.

\section{Kernel localization for non-auxiliary LMM systems}\label{sec:kernel}
For A-B schemes with $M\geq 2$ and A-M schemes with $M\geq 1$, the non-auxiliary system \eqref{eq:Bh_system} is underdetermined. However, as observed in \Cref{fig:A-B--A-M}, this nonuniqueness is not distributed over the whole time grid. It is mainly confined to the left boundary for A-B schemes, and to both endpoints for A-M schemes, while the recovered grid values in the interior are nearly unchanged. In this section, we explain this phenomenon from the viewpoint of the kernel of the corresponding discovery matrix. Since the difference between any two grid solutions of the non-auxiliary system belongs to this kernel, the spatial distribution of the nonuniqueness is determined by the structure of its kernel vectors. We prove that, for A-B schemes, these vectors decay exponentially away from the initial indices, whereas for A-M schemes they decay exponentially away from both endpoints into the interior. We first present an elementary spectral estimate that will be used in the kernel analysis of the non-auxiliary A-B and A-M discovery matrices.

\subsection{A spectral estimate}

We first present the spectral decomposition used in the following kernel analysis. Let $\BC\in\R^{q\times q}$, and let $\sigma(\BC)\subset\mathbb C$ denote the spectrum of $\BC$. When $\BC$ has no eigenvalue on the unit circle, i.e., $\sigma(\BC)\cap\{\lambda\in\mathbb C:|\lambda|=1\}=\emptyset$, the spectrum is decomposed into the stable and unstable parts
\begin{equation*}
\sigma_s=\{\lambda\in\sigma(\BC):|\lambda|<1\},\qquad \sigma_u=\{\lambda\in\sigma(\BC):|\lambda|>1\}.
\end{equation*}
Define $E_s$ and $E_u$ as the real subspaces generated by the real and imaginary parts of generalized eigenvectors associated with eigenvalues in $\sigma_s$ and $\sigma_u$, respectively. Then
\begin{equation*}
\R^q=E_s\oplus E_u.
\end{equation*}

With this direct-sum decomposition, every vector $x\in\R^q$ can be written uniquely as $x=x_s+x_u$, where $x_s\in E_s$ and $x_u\in E_u$. We define the associated projections $P_s$ and $P_u$ by
\begin{equation*}
P_sx=x_s,\qquad P_ux=x_u.
\end{equation*}
Thus, $P_s$ extracts the stable component of $x$, while $P_u$ extracts the unstable component of $x$.

Moreover, since each generalized eigenspace of $\BC$ is invariant under $\BC$, the subspaces $E_s$ and $E_u$ are invariant under $\BC$, i.e., $\BC E_s\subset E_s$ and $\BC E_u\subset E_u$. Therefore, for every $v_s\in E_s$ and $v_u\in E_u$, we can define the restricted operators $\BC|_{E_s}:E_s\to E_s$ and $\BC|_{E_u}:E_u\to E_u$ by
\begin{equation*}
(\BC|_{E_s})v_s=\BC v_s,\qquad (\BC|_{E_u})v_u=\BC v_u.
\end{equation*}
If $E_u\neq\{0\}$, then all eigenvalues of $\BC|_{E_u}$ satisfy $|\lambda|>1$. Hence  $\BC|_{E_u}$ is invertible on $E_u$.


The following elementary estimate translates the spectral stability of a matrix into exponential decay estimates. These estimates will be used to prove the boundary localization of the nonuniqueness.

\begin{lemma}\label{lem:spectral_power}
Let $\BC\in\R^{q\times q}$. Suppose that for some $\rho\geq0$, $\sigma(\BC)\subset \{\lambda\in\mathbb C:|\lambda|\leq \rho\}$. Then, for every $\theta>\rho$, there exists a constant $K>0$ independent of $n$ such that
\begin{equation*}
\|\BC^n\|_2\leq K\theta^n,\qquad n=0,1,\ldots.
\end{equation*}

Moreover, assume $\sigma(\BC)\cap \{\lambda\in\mathbb C:|\lambda|=1\}=\emptyset$ and $E_u\neq\{0\}$. Set $\max\emptyset:=0$ and define $\hat{\rho}=\max\left\{\max_{\lambda\in\sigma_s}|\lambda|,\max_{\lambda\in\sigma_u}|\lambda|^{-1}\right\}<1$. Then, for every $\theta\in(\hat{\rho},1)$, there exists a constant $\tilde{K}>0$ independent of $n$ such that
\begin{equation*}
\|\BC^nP_s\|_2+\|(\BC|_{E_u})^{-n}P_u\|_2\leq \tilde{K}\theta^n,\qquad n=0,1,\ldots.
\end{equation*}
\end{lemma}
\begin{proof}
Let $\BC=SJS^{-1}$ be a Jordan decomposition of $\BC$, where $J$ is the Jordan form of $\BC$ and $S$ is nonsingular. For each Jordan block $J_\lambda$ of $J$ associated with an eigenvalue $\lambda\in\sigma(\BC)$, suppose that $J_\lambda$ has size $m$. Then, $J_\lambda=\lambda I_m+Q$, where $Q$ is a nilpotent matrix with $Q^m=0$. Hence, for every $n\in\{0,1,\ldots\}$, we have
\begin{equation*}
J_\lambda^n=\sum_{\ell=0}^{\min\{n,m-1\}}\binom{n}{\ell}\lambda^{n-\ell}Q^\ell .
\end{equation*}
Since $|\lambda|\leq \rho<\theta$, for each fixed $\ell=0,\ldots,m-1$ and all $n\geq \ell$, the sequence $\binom{n}{\ell}|\lambda|^{n-\ell}\theta^{-n}$ is bounded. Thus, for every Jordan block, we have
\begin{equation*}
\|J_\lambda^n\|_2\leq \left(\sum_{\ell=0}^{m-1}\sup_{n\geq \ell}\binom{n}{\ell}|\lambda|^{n-\ell}\theta^{-n}\|Q^\ell\|_2\right)\theta^n,\qquad n=0,1,\ldots.
\end{equation*}
Since $J$ contains a finite number of Jordan blocks, it follows that $\|J^n\|_2\leq \hat{K}\theta^n$, where $\hat{K}$ only depends on the matrix $\BC$. Therefore, there exists a constant $K$ independent of $n$ such that
\begin{equation}\label{eq:lemma3.1-1}
\|\BC^n\|_2\leq \|S\|_2\|S^{-1}\|_2\|J^n\|_2\leq K\theta^n.
\end{equation}

For the next estimate, if $E_u\neq\{0\}$, then $\BC|_{E_u}$ is invertible and $\sigma((\BC|_{E_u})^{-1})=\{\lambda^{-1}:\lambda\in\sigma(\BC),\ |\lambda|>1\}$. Thus, by the definition of $\hat{\rho}$, we have
\begin{equation*}
r(\BC|_{E_s})=\max_{\lambda\in\sigma_s}|\lambda|\leq \hat{\rho}<\theta \quad \text{and} \quad r((\BC|_{E_u})^{-1})=\max_{\lambda\in\sigma_u}|\lambda|^{-1}\leq \hat{\rho}<\theta,
\end{equation*}
where $r(\cdot)$ denotes the spectral radius. Therefore, applying \eqref{eq:lemma3.1-1} to $\BC|_{E_s}$ and $(\BC|_{E_u})^{-1}$, there exist constants $K_s,K_u>0$ such that
\begin{equation*}
\|(\BC|_{E_s})^n\|_2\leq K_s\theta^n,\qquad \|(\BC|_{E_u})^{-n}\|_2\leq K_u\theta^n.
\end{equation*}
Finally, since $\|P_s\|_2$ and $\|P_u\|_2$ only depend on the matrix $\BC$, we obtain
\begin{equation*}
\|\BC^nP_s\|_2+\|(\BC|_{E_u})^{-n}P_u\|_2\leq \tilde{K}\theta^n,\qquad n=0,1,2,\ldots,
\end{equation*}
where $\tilde{K}$ is independent of $n$.
\end{proof}

\subsection{Adams--Bashforth schemes}
We first consider the non-auxiliary discovery system generated by an $M$-step A-B scheme. Since $\beta_0=0$, the unknown grid vector is $\Bf_h=(f_0,f_1,\ldots,f_{N-1})^\top$. In this case, \eqref{eq:Bh_system} takes the form
\begin{equation}\label{eq:ab_equations}
\sum_{m=1}^M \beta_m f_{n-m}=q_n,\qquad n=M,\ldots,N.
\end{equation}
The relevant discovery polynomial is
\begin{equation}\label{eq:p_ab}
p_{\rm AB}(z)=\beta_1z^{M-1}+\beta_2z^{M-2}+\cdots+\beta_{M-1}z+\beta_M.
\end{equation}
We say that the A-B discovery system satisfies the discovery-stability condition if 
\begin{equation}\label{eq:ab_discovery_stability}
\rho_{\rm AB}:=\max\{|\lambda|:\lambda\in\mathbb C,\ p_{\rm AB}(\lambda)=0\}<1.
\end{equation}

Let $\Bf_h^{(1)}$ and $\Bf_h^{(2)}$ be any two grid solutions of the non-auxiliary A-B system \eqref{eq:ab_equations}, and define their difference by
\begin{equation*}
\Bd=\Bf_h^{(1)}-\Bf_h^{(2)}=(d_0,\ldots,d_{N-1})^\top .
\end{equation*}
Subtracting the two systems yields
\begin{equation}\label{eq:ab_kernel_equations}
\sum_{m=1}^M \beta_m d_{n-m}=0,\qquad n=M,\ldots,N.
\end{equation}
Equivalently, this homogeneous system \eqref{eq:ab_kernel_equations} can be written as $\BB_h \Bd=0$. Thus $\Bd\in\Ker(\BB_h)$. Therefore, the kernel of $\BB_h$ describes the possible differences between grid solutions. 

The following result shows that the difference between any two grid solutions of the non-auxiliary A-B system is determined by the first few components and decays exponentially away from the initial indices.

\begin{theorem}\label{thm:ab_kernel}
For the A-B scheme with $M=1$, $\Ker(\BB_h)=\{0\}$. Let $M\geq 2$ and suppose that the A-B discovery-stability condition \eqref{eq:ab_discovery_stability} holds. Then, for every $\theta\in(\rho_{\rm AB},1)$, there exists a constant $C>0$, independent of $N$, such that every $\Bd=(d_0,\ldots,d_{N-1})^\top\in\Ker(\BB_h)$ satisfies
\begin{equation}\label{eq:ab_kernel_decay}
|d_i|\leq C\theta^i\sum_{j=0}^{M-2}|d_j|,\qquad i=0,\ldots,N-1.
\end{equation}
\end{theorem}

\begin{proof}
Note that $\beta_1\neq 0$ in the A-B schemes. For $M=1$ and $\forall~\Bd=(d_0,\ldots,d_{N-1})^\top\in\Ker(\BB_h)$, according to \eqref{eq:ab_kernel_equations} and $\beta_1\neq 0$, we obtain $\beta_1d_{n-1}=0$ for $n=1,\ldots,N$. Thus $\Bd=0$. 

For $M\geq 2$, \eqref{eq:ab_kernel_equations} gives
\begin{equation*}
d_{k+M-1}=-\frac{\beta_2}{\beta_1}d_{k+M-2}-\frac{\beta_3}{\beta_1}d_{k+M-3}-\cdots-\frac{\beta_M}{\beta_1}d_k,
\end{equation*}
where $k=0,\ldots,N-M$. Let
\begin{equation*}
\By_k:=\begin{bmatrix}d_k&d_{k+1}&\cdots&d_{k+M-2}\end{bmatrix}^\top.
\end{equation*}
Then $\By_{k+1}=\BC_{\rm AB}\By_k$, where
\begin{equation}\label{eq:matrix-C_AB}
\BC_{\rm AB}=\begin{bmatrix}
0&1&0&\cdots&0\\
0&0&1&\cdots&0\\
\vdots&\vdots&\vdots&\ddots&\vdots\\
0&0&0&\cdots&1\\
-\beta_M/\beta_1&-\beta_{M-1}/\beta_1&-\beta_{M-2}/\beta_1&\cdots&-\beta_2/\beta_1
\end{bmatrix}.
\end{equation}
The characteristic polynomial of $\BC_{\rm AB}$ is
\begin{equation*}
\det(zI_{M-1}-\BC_{\rm AB})=z^{M-1}+\frac{\beta_2}{\beta_1}z^{M-2}+\cdots+\frac{\beta_{M-1}}{\beta_1}z+\frac{\beta_M}{\beta_1}=\frac{p_{\rm AB}(z)}{\beta_1}.
\end{equation*}
Since $\beta_1\neq0$, it follows that
\begin{equation}\label{eq:sigma(BC)}
\sigma(\BC_{\rm AB})=\{\lambda\in\mathbb C:\det(\lambda I_{M-1}-\BC_{\rm AB})=0\}=\{\lambda\in\mathbb C:p_{\rm AB}(\lambda)=0\}.
\end{equation}
By \Cref{lem:spectral_power}, for the fixed $\theta\in(\rho_{\rm AB},1)$, we have
\begin{equation}\label{eq:Y_k}
\|\By_k\|_2\leq \hat{C}\theta^k\|\By_0\|_2,
\end{equation}
where $\hat{C}>0$ is a constant independent of $k$ and $N$. 

For $0\leq i\leq M-2$, we have
\begin{equation}\label{eq:d_i1}
|d_i|\leq \sum_{j=0}^{M-2}|d_j|\leq \theta^{-(M-2)}\theta^i\sum_{j=0}^{M-2}|d_j|.
\end{equation}
For $i\geq M-1$, let $k=i-M+2$. Then $d_i$ is the last element of $\By_k$. Using \eqref{eq:Y_k}, we obtain
\begin{equation}\label{eq:d_i2}
|d_i|\leq \|\By_k\|_2\leq \hat{C}\theta^k\sum_{j=0}^{M-2}|d_j|=\hat{C}\theta^{-M+2}\theta^{i}\sum_{j=0}^{M-2}|d_j|.
\end{equation}

Finally, combining \eqref{eq:d_i1} and \eqref{eq:d_i2} yields \eqref{eq:ab_kernel_decay} for $i=0,\ldots,N-1$ where $C= \max\{\theta^{-(M-2)},\hat{C}\theta^{-M+2}\}$.
\end{proof}

For an $M$-step A-B scheme, the localization condition in \Cref{thm:ab_kernel} is a root condition on the discovery polynomial \eqref{eq:p_ab}. The root radii $\rho_{\rm AB}=\max\{|\lambda|:p_{\rm AB}(\lambda)=0\}$ for the A-B schemes are listed in \Cref{tab:ab-root-radii}. The A-B schemes with $M=1,\ldots,6$ satisfy the discovery-stability condition, while $M=7$ does not.
\begin{table}[htbp!]
\centering
\caption{Root radii for A-B schemes.}
\label{tab:ab-root-radii}
\begin{tabular}{cccccccc}
\toprule
$M$ & 1 & 2 & 3 & 4 & 5 & 6 & 7\\
\midrule
$\rho_{\rm AB}$ & 0 & 0.3333 & 0.4663 & 0.6338 & 0.8075 & 0.9829 & 1.1587\\
\bottomrule
\end{tabular}
\end{table}

This theorem identifies the algebraic source of the localization observed in \Cref{fig:A-B--A-M} about the non-auxiliary A-B schemes. When $M \geq 2$, the LMM residual does not uniquely determine all involved grid values. Therefore, the difference between any two zero-residual grid solutions belongs to $\Ker(\BB_h)$. The estimate \eqref{eq:ab_kernel_decay} shows that every such kernel vector is controlled only by its first $M-1$ components and then decays exponentially away from the initial endpoint. This explains why the zero-residual A-B solutions in \Cref{fig:A-B--A-M} may differ near the first few grid points but become nearly indistinguishable in the interior and near the terminal time. Thus, the observed boundary-layer behavior is a consequence of the algebraic kernel structure of the non-auxiliary A-B discovery matrix, rather than a feature of a particular optimization algorithm or a special numerical solution.

We next present a nonhomogeneous version of this estimate, which will be used in \Cref{sec:generalization}. It applies when a grid solution of the discovery system is compared with a prescribed grid sequence evaluated on the sampled trajectory. In the trajectory-error analysis below, this prescribed sequence is the exact vector field sampled along the trajectory, and the forcing term is the local truncation residual.

Let $\widetilde{\Bf}_h=(\widetilde f_0,\ldots,\widetilde f_{N-1})^\top$ be a grid solution of the non-auxiliary A-B system \eqref{eq:ab_equations}, and let $\Bv_h=(v_0,\ldots,v_{N-1})^\top$ be a prescribed grid sequence. Define the residual sequence $\Br_h=(r_M,\ldots,r_N)^\top$ by
\begin{equation}\label{eq:ab_residual_def}
r_n:=q_n-\sum_{m=1}^M\beta_m v_{n-m},\qquad n=M,\ldots,N.
\end{equation}
Set $\Be_h=\widetilde{\Bf}_h-\Bv_h=(e_0,\ldots,e_{N-1})^\top$. Subtracting \eqref{eq:ab_residual_def} from \eqref{eq:ab_equations} gives
\begin{equation}\label{eq:forced_ab}
\sum_{m=1}^M\beta_m e_{n-m}=r_n,\qquad n=M,\ldots,N.
\end{equation}
If $\Bv_h$ is also a solution of the same non-auxiliary A-B system, then $\Br_h=0$ and $\Be_h$ is precisely a kernel vector of the form considered in \Cref{thm:ab_kernel}.

\begin{lemma}\label{lem:forced_ab}
For the A-B scheme with $M=1$, the error sequence $\Be_h$ satisfies
\begin{equation}\label{eq:forced_ab_M1}
|e_i|\leq |\beta_1|^{-1}|r_{i+1}|,\qquad i=0,\ldots,N-1.
\end{equation}
Let $M\geq 2$ and suppose that the A-B discovery-stability condition \eqref{eq:ab_discovery_stability} holds. Then, for every $\theta\in(\rho_{\rm AB},1)$, there exists a constant $C>0$, independent of $N$, such that
\begin{equation}\label{eq:forced_ab_estimate}
|e_i|\leq C\left(\theta^i\sum_{j=0}^{M-2}|e_j|+\sum_{\ell=M}^{\min\{N,i+1\}}\theta^{i+1-\ell}|r_\ell|\right),\qquad i=0,\ldots,N-1,
\end{equation}
where the residual sum is zero if $i+1<M$.
\end{lemma}

\begin{proof}
For $M=1$, according to \eqref{eq:forced_ab} and $\beta_1\neq0$, we obtain $\beta_1 e_{n-1}=r_n$ for $n=1,\ldots,N$. Therefore, \eqref{eq:forced_ab_M1} holds.

For $M\geq2$, we use the same argument as in the proof of \Cref{thm:ab_kernel}. From \eqref{eq:forced_ab}, we have
\begin{equation*}
e_{k+M-1}=-\frac{\beta_2}{\beta_1}e_{k+M-2}-\frac{\beta_3}{\beta_1}e_{k+M-3}-\cdots-\frac{\beta_M}{\beta_1}e_k+\frac{r_{k+M}}{\beta_1},
\end{equation*}
where $k=0,\ldots,N-M$. Define $\By_k=[e_k~e_{k+1}~\cdots~e_{k+M-2}]^\top$. Then
\begin{equation}\label{eq:forced_state_ab}
\By_{k+1}=\BC_{\rm AB}\By_k+\frac{r_{k+M}}{\beta_1}\mathbf e_{M-1},
\end{equation}
where $\BC_{\rm AB}$ is defined by \eqref{eq:matrix-C_AB} and $\mathbf e_{M-1}=(0,\ldots,0,1)^\top\in\R^{M-1}$ denotes the last coordinate vector. Iterating \eqref{eq:forced_state_ab} gives
\begin{equation*}
\By_k=\BC_{\rm AB}^k\By_0+\sum_{s=0}^{k-1}\BC_{\rm AB}^{k-1-s}\frac{r_{s+M}}{\beta_1}\mathbf e_{M-1},
\end{equation*}
for $k=1,\ldots,N-M+1$. According to \eqref{eq:sigma(BC)}, the spectrum of $\BC_{\rm AB}$ is given by the roots of $p_{\rm AB}$. By \Cref{lem:spectral_power}, for every $\theta\in(\rho_{\rm AB},1)$, there exists a constant $\hat C>0$, independent of $N$, such that
\begin{equation}\label{eq:forced_Yk_bound}
\|\By_k\|_2\leq \hat C\Big(\theta^k\|\By_0\|_2+\sum_{s=0}^{k-1}\theta^{k-1-s}|r_{s+M}|\Big),\qquad k=0,\ldots,N-M+1.
\end{equation}

For $0\leq i\leq M-2$, the residual sum in \eqref{eq:forced_ab_estimate} is empty, and
\begin{equation}\label{eq:case1}
|e_i|\leq \sum_{j=0}^{M-2}|e_j|\leq \theta^{-(M-2)}\theta^i\sum_{j=0}^{M-2}|e_j|.
\end{equation}
For $i\geq M-1$, let $k=i-M+2$. Then $e_i$ is the last component of $\By_k$. Using \eqref{eq:forced_Yk_bound} and $\|\By_0\|_2\leq\sum_{j=0}^{M-2}|e_j|$, we obtain
\begin{equation}\label{eq:case2}
\begin{aligned}
|e_i|\leq \|\By_k\|_2&\leq \hat C\Big(\theta^{-M+2}\theta^i\sum_{j=0}^{M-2}|e_j|+\sum_{\ell=M}^{i+1}\theta^{i+1-\ell}|r_\ell|\Big).
\end{aligned}
\end{equation}
Since $i\leq N-1$, the last residual sum agrees with the sum in \eqref{eq:forced_ab_estimate}. Combining \eqref{eq:case1} and \eqref{eq:case2} yields \eqref{eq:forced_ab_estimate}.
\end{proof}


\subsection{Adams--Moulton schemes}

We next consider the non-auxiliary discovery system generated by an $M$-step A-M scheme. In contrast to the A-B scheme, here $\beta_0\neq0$, and the unknown grid vector is $\Bf_h=(f_0,f_1,\ldots,f_N)^\top$. In this case, \eqref{eq:Bh_system} can be rewritten as
\begin{equation}\label{eq:am_equations}
\sum_{m=0}^M \beta_m f_{n-m}=q_n,\qquad n=M,\ldots,N.
\end{equation}
The corresponding discovery polynomial is
\begin{equation*}
p_{\rm AM}(z)=\beta_0z^M+\beta_1z^{M-1}+\cdots+\beta_{M-1}z+\beta_M.
\end{equation*}
We say that the A-M discovery system satisfies the discovery-stability condition if
\begin{equation}\label{eq:am_discovery-stability}
p_{\rm AM}(\lambda)\neq0
\qquad \text{for every }\lambda\in\mathbb C\text{ with }|\lambda|=1.
\end{equation}
Under this condition, let $\max\emptyset=0$ and
\begin{equation*}
\rho_{\rm AM}:=\max\left\{\max_{p_{\rm AM}(\lambda)=0,\ |\lambda|<1}|\lambda|,\max_{p_{\rm AM}(\lambda)=0,\ |\lambda|>1}|\lambda|^{-1}\right\}.
\end{equation*}

Let $\Bf_h^{(1)}$ and $\Bf_h^{(2)}$ be any two grid solutions of the non-auxiliary A-M system \eqref{eq:am_equations}, and set $\Bd=\Bf_h^{(1)}-\Bf_h^{(2)}=(d_0,\ldots,d_N)^\top$. Subtracting the two systems gives
\begin{equation}\label{eq:am_kernel_equations}
\sum_{m=0}^M \beta_m d_{n-m}=0,\qquad n=M,\ldots,N.
\end{equation}
Equivalently, \eqref{eq:am_kernel_equations} can be written as $\BB_h\Bd=0$. Thus $\Bd\in\Ker(\BB_h)$. As in the A-B scheme, the kernel of $\BB_h$ describes the possible differences between any two grid solutions. The main difference from the A-B scheme is that the kernel behavior of the A-M system is controlled by both endpoint indices.

\begin{theorem}\label{thm:am_kernel}
Suppose that the A-M discovery-stability condition \eqref{eq:am_discovery-stability} holds. Then, for every $\theta\in(\rho_{\rm AM},1)$, there exists a constant $C>0$, independent of $N$, such that $\Bd=(d_0,\ldots,d_N)^\top\in\Ker(\BB_h)$ satisfies
\begin{equation}\label{eq:am_kernel_decay}
|d_i|\leq C\left(\theta^i\sum_{j=0}^{M-1}|d_j|+\theta^{N-i}\sum_{j=N-M+1}^{N}|d_j|\right),\qquad i=0,\ldots,N.
\end{equation}
\end{theorem}

\begin{proof}
Note that $\beta_0\neq 0$ in the A-M schemes. For $\Bd\in\Ker(\BB_h)$, \eqref{eq:am_kernel_equations} gives
\begin{equation*}
d_{k+M}=-\frac{\beta_1}{\beta_0}d_{k+M-1}-\frac{\beta_2}{\beta_0}d_{k+M-2}-\cdots-\frac{\beta_M}{\beta_0}d_{k},
\end{equation*}
where $k=0,\ldots,N-M$. Let $\By_k:=[d_k~d_{k+1}~\cdots~d_{k+M-1}]^{\top}$. Then $\By_{k+1}=\BC_{\rm AM}\By_k$, where
\begin{equation}\label{eq:matrix-C_AM}
\BC_{\rm AM}=
\begin{bmatrix}
0&1&0&\cdots&0\\
0&0&1&\cdots&0\\
\vdots&\vdots&\vdots&\ddots&\vdots\\
0&0&0&\cdots&1\\
-\beta_M/\beta_0&-\beta_{M-1}/\beta_0&-\beta_{M-2}/\beta_0&\cdots&-\beta_1/\beta_0
\end{bmatrix}.
\end{equation}
The characteristic polynomial of $\BC_{\rm AM}$ is
\begin{equation*}
\det(zI_M-\BC_{\rm AM})=z^M+\frac{\beta_1}{\beta_0}z^{M-1}+\cdots+\frac{\beta_{M-1}}{\beta_0}z+\frac{\beta_M}{\beta_0}=\frac{p_{\rm AM}(z)}{\beta_0}.
\end{equation*}
Therefore,
\begin{equation}\label{eq:sigma_C_AM}
\sigma(\BC_{\rm AM})=\{\lambda\in\mathbb C:p_{\rm AM}(\lambda)=0\}.
\end{equation}
By \eqref{eq:am_discovery-stability} and \eqref{eq:sigma_C_AM}, we have $\sigma(\BC_{\rm AM})\cap\{\lambda\in\mathbb C:|\lambda|=1\}=\emptyset$. Recall that $\mathbb R^M=E_s\oplus E_u$ is the stable-unstable decomposition associated with the roots inside and outside the unit circle, and $P_s$ and $P_u$ are the corresponding spectral projections. Since $\By_{N-M+1}=\BC_{\rm AM}^{N-M+1} \By_0$, we have
\begin{equation*}
P_u\By_{N-M+1}=P_u\BC_{\rm AM}^{N-M+1}\By_0=(\BC_{\rm AM}|_{E_u})^{N-M+1} P_u\By_0.
\end{equation*}
Hence, for $k=0,\ldots,N-M+1$, we have
\begin{equation*}
\BC_{\rm AM}^kP_u\By_0=(\BC_{\rm AM}|_{E_u})^{-(N-M+1-k)}P_u\By_{N-M+1}.
\end{equation*}
Therefore, using $\By_k=\BC_{\rm AM}^k\By_0$ and $I=P_s+P_u$, we obtain the representation
\begin{equation*}
\By_k=\BC_{\rm AM}^{k}P_s\By_0+(\BC_{\rm AM}|_{E_u})^{-(N-M+1-k)}P_u\By_{N-M+1}.
\end{equation*}
If $E_u=\{0\}$, the second term is omitted.

According to \eqref{eq:sigma_C_AM} and using \Cref{lem:spectral_power}, we obtain
\begin{equation}\label{eq:am_Yk_bound}
\|\By_k\|_2\leq\hat C\left(\theta^k\|\By_0\|_2+\theta^{N-M+1-k}\|\By_{N-M+1}\|_2\right),
\end{equation}
where $\theta\in(\rho_{\rm AM},1)$ and $\hat C>0$ is independent of $N$.

For $0\leq i\leq M-2$, the estimate follows
\begin{equation}\label{eq:thm:am_kernel-01}
|d_i|\leq \sum_{j=0}^{M-1}|d_j|\leq \theta^{-(M-1)}\theta^i \sum_{j=0}^{M-1}|d_j|.
\end{equation}
For $i\geq M-1$, set $k=i-M+1$. Then $0\leq k\leq N-M+1$ and $d_i$ is the last component of $\By_k$. By \eqref{eq:am_Yk_bound}, we have
\begin{equation}\label{eq:thm:am_kernel-02}
|d_i|\leq\|\By_k\|_2\leq\hat C\Bigg(\theta^{i-M+1}\sum_{j=0}^{M-1}|d_j|+\theta^{N-i}\sum_{j=N-M+1}^{N}|d_j|\Bigg).
\end{equation}
Combining \eqref{eq:thm:am_kernel-01} and \eqref{eq:thm:am_kernel-02} gives \eqref{eq:am_kernel_decay}.
\end{proof}

The values of $\rho_{\rm AM}$ for several classical A-M schemes are listed in \Cref{tab:am-dichotomy}. For the one-step A-M scheme, the discovery polynomial has the root $-1$, which violates the hypothesis of \Cref{thm:am_kernel}.
\begin{table}[htbp!]
\centering
\caption{Root radii for A-M schemes.}
\label{tab:am-dichotomy}
\begin{tabular}{cccccc}
\toprule
$M$ & 1 & 2 & 3 & 4 & 5\\
\midrule
$\rho_{\rm AM}$ & 1 & 0.5826 & 0.4227 & 0.3359 & 0.4530\\
\bottomrule
\end{tabular}
\end{table}

This theorem provides a corresponding explanation for the results presented in \Cref{fig:A-B--A-M} for non-auxiliary A-M schemes. The estimate \eqref{eq:am_kernel_decay} shows that each kernel vector is controlled by $M$ left endpoint values and $M$ right endpoint values. The influence of the initial endpoint decays exponentially from the left, while the influence of the terminal endpoint decays exponentially from the right. This explains the two-sided boundary layer behavior observed in \Cref{fig:A-B--A-M}. Therefore, the localization phenomenon is determined by the algebraic kernel structure of the non-auxiliary A-M discovery matrix, rather than by a specific optimization method or a special numerical solution.

We next state the corresponding nonhomogeneous estimate, in the same spirit as \Cref{lem:forced_ab}. It applies when a grid solution of the discovery system is compared with a prescribed grid sequence that satisfies the LMM equations only up to a residual. In \Cref{sec:generalization}, the prescribed sequence will be the exact vector field sampled along the trajectory, and the residual will be the local truncation error.

Let $\widetilde{\Bf}_h=(\widetilde f_0,\ldots,\widetilde f_N)^\top$ be a grid solution of the non-auxiliary A-M system \eqref{eq:am_equations}, and let $\Bv_h=(v_0,\ldots,v_N)^\top$ be a prescribed grid sequence. Define the residual sequence $\Br_h=(r_M,\ldots,r_N)^\top$ by
\begin{equation}\label{eq:am_residual_def}
r_n:=q_n-\sum_{m=0}^M\beta_m v_{n-m},\qquad n=M,\ldots,N.
\end{equation}
Set $\Be_h=\widetilde{\Bf}_h-\Bv_h=(e_0,\ldots,e_N)^\top$. Subtracting \eqref{eq:am_residual_def} from \eqref{eq:am_equations} gives
\begin{equation}\label{eq:forced_am}
\sum_{m=0}^M\beta_m e_{n-m}=r_n,\qquad n=M,\ldots,N.
\end{equation}
If $\Bv_h$ is also a grid solution of the same non-auxiliary A-M system, then $\Br_h=0$, and $\Be_h$ reduces to a kernel vector considered in \Cref{thm:am_kernel}.

\begin{lemma}\label{lem:forced_am}
Suppose that the A-M discovery-stability condition \eqref{eq:am_discovery-stability} holds. Let $\Be_h=(e_0,\ldots,e_N)^\top$ satisfy \eqref{eq:forced_am}. Then, for every $\theta\in(\rho_{\rm AM},1)$, there exists a constant $C>0$, independent of $N$, such that
\begin{equation}\label{eq:forced_am_estimate}
|e_i|\leq C\left(\theta^i\sum_{j=0}^{M-1}|e_j|+\theta^{N-i}\sum_{j=N-M+1}^{N}|e_j|+\sum_{\ell=M}^{N}\theta^{|i-\ell|}|r_\ell|\right),\qquad i=0,\ldots,N.
\end{equation}
\end{lemma}

\begin{proof}
Define $\By_k=[e_k~e_{k+1}~\cdots~e_{k+M-1}]^\top$ for $k=0,\ldots,N-M+1$. According to \eqref{eq:forced_am}, we obtain 
\begin{equation*}
\By_{k+1}=\BC_{\rm AM}\By_k+\frac{r_{k+M}}{\beta_0}\mathbf e_M,\qquad k=0,\ldots,N-M,
\end{equation*}
where $\BC_{\rm AM}$ is defined by \eqref{eq:matrix-C_AM}, and $\mathbf e_M=(0,\ldots,0,1)^\top\in\mathbb R^M$ denotes the last coordinate vector.

With the same stable-unstable decomposition $\mathbb R^M=E_s\oplus E_u$ as in the proof of \Cref{thm:am_kernel}, we get
\begin{equation*}
\By_k=\BC_{\rm AM}^{k}P_s\By_0+(\BC_{\rm AM}|_{E_u})^{-(N-M+1-k)}P_u\By_{N-M+1}+\sum_{\ell=0}^{N-M}
G_{k,\ell}\frac{r_{\ell+M}}{\beta_0}\mathbf e_M,
\end{equation*}
where
\begin{equation*}
G_{k,\ell}=
\begin{cases}
\BC_{\rm AM}^{k-1-\ell}P_s, & \ell<k,\\
-(\BC_{\rm AM}|_{E_u})^{-(\ell+1-k)}P_u, & \ell\geq k.
\end{cases}
\end{equation*}
The terms involving $E_u$ are omitted if $E_u=\{0\}$. By \Cref{lem:spectral_power}, we have
\begin{equation*}
\|G_{k,\ell}\|_2\leq\hat C\theta^{|k-\ell|},
\end{equation*}
where $0\leq k\leq N-M+1$, $0\leq \ell\leq N-M$ and $\hat C$ independent of $N$. Therefore,
\begin{equation}\label{eq:forced_am_Yk_bound}
\|\By_k\|_2\leq C\Big(\theta^k\|\By_0\|_2+\theta^{N-M+1-k}\|\By_{N-M+1}\|_2+\sum_{\ell=0}^{N-M}\theta^{|k-\ell|}|r_{\ell+M}|\Big).
\end{equation}

For $0\leq i\leq M-2$, the desired estimate follows directly from the left endpoint term. For $M-1\leq i\leq N$, set $k=i-M+1$. Then $0\leq k\leq N-M+1$, and $e_i$ is the last component of $\By_k$. From \eqref{eq:forced_am_Yk_bound}, we obtain
\begin{equation*}
|e_i|\leq C\left(\theta^i\sum_{j=0}^{M-1}|e_j|+\theta^{N-i}\sum_{j=N-M+1}^{N}|e_j|
+\sum_{\ell=0}^{N-M}\theta^{|i-M+1-\ell|}|r_{\ell+M}|\right).
\end{equation*}
Setting $s=\ell+M$ in the forcing sum yields $\theta^{|i+1-s|}$, which is bounded by a constant multiple of $\theta^{|i-s|}$. Absorbing this factor into $C$ proves \eqref{eq:forced_am_estimate}.
\end{proof}

\section{Whole-trajectory generalization}\label{sec:generalization}
In this section, we lift error estimates on the observation grid to whole-trajectory generalization estimates. Clearly, grid error estimates cannot control continuous in time errors along the observed trajectory. Let $f$ be an arbitrary component of $\Bf$, and let $u_h\in\calA$ be a learned scalar approximation of $f$. This componentwise treatment is justified because the coefficients of a LMM act only on the time index and do not mix the output components of the vector field. Thus, the following scalar estimates can be applied to each component separately, while the trajectory $\Bx(t)$ itself remains the full $d$-dimensional observed trajectory.

Once the trajectory $\Bx(t)$ is fixed, the error of a learned approximation $u_h\in\calA$ along the trajectory is the one-dimensional trace
\begin{equation}\label{eq:error}
g_h(t)=u_h(\Bx(t))-f(\Bx(t)),\qquad t\in[0,T].
\end{equation}
Thus, in the sense considered here, whole-trajectory generalization is a deterministic one-dimensional sampling problem along the observed trajectory, rather than a high-dimensional statistical learning problem. We derive such whole-trajectory estimates for both the auxiliary and non-auxiliary formulations.

\subsection{Sampling inequalities along the trajectory}
We use standard one-dimensional sampling inequalities (see \Cref{lem:sampling_l}) to convert grid error estimates into continuous-time estimates along the observation trajectory. Such inequalities are classical in the analysis of Sobolev functions sampled on discrete point sets; see \cite{Madych2006,Rieger2008}. For the A-B and A-M discovery systems considered below, the involved sampling set $\{t_n:n\in\calI_h\}$ is a quasi-uniform subset of $[0,T]$. In the A-B case, the terminal endpoint may be missing, while in the A-M case the full set of endpoints is included. Since the step number $M$ is fixed, these endpoint differences only affect the constants and do not change the order of the estimates.

\begin{lemma}\label{lem:sampling_l}
Let $p\geq1$ be an integer and let $\calI_h=\{s,s+1,\ldots,e(N)\}$ be one of the sets of time indices involved in the A-B or A-M discovery system with fixed step number $M$. For $h$ sufficiently small, the following estimates hold. If $v\in H^p(0,T)$, then
\begin{equation}\label{eq:sampling_l2}
\|v\|_{L^2(0,T)}\leq C\left[\left(h\sum_{n\in\calI_h}|v(t_n)|^2\right)^{\frac 12}+h^p|v|_{H^p(0,T)}\right].
\end{equation}
Moreover, if $v\in W^{p,1}(0,T)$, then
\begin{equation}\label{eq:sampling_l1}
\|v\|_{L^1(0,T)}\leq C\left[h\sum_{n\in\calI_h}|v(t_n)|+h^p|v|_{W^{p,1}(0,T)}\right].
\end{equation}
Here the constant $C$ is independent of $h$ and $v$.
\end{lemma}

The same estimation method can also be applied to fixed subintervals of $[0,T]$ where the restricted sample set remains quasi-uniform, and the constant remains independent of $h$.

\subsection{Trajectory generalization for the auxiliary formulation}\label{sec:aux_generalization}
We first show how the grid error estimate for the auxiliary formulation can be converted into a continuous-time error estimate along the observed trajectory. 

For $p\geq1$ and $\tilde{R}>0$, define
\begin{equation}\label{eq:trajectory_regular_class}
\calA_{\tilde{R}}=\left\{u_h\in\calA:\ g_h\in H^p(0,T),\ |g_h|_{H^p(0,T)}\leq \tilde{R}\right\},
\end{equation}
where $g_h$ is defined by \eqref{eq:error}. This set collects functions whose trace errors along the observed trajectory are uniformly bounded in $H^p(0,T)$.

\begin{theorem}\label{thm:aux_generalization}
In the dynamical system \eqref{eq:intro_ode}, suppose $\Bx\in C^{\infty}([0,T])^d$ and $\Bf$ is defined in a small neighborhood of $\calT$. Assume that $\hf_{\calA,h}$ is a global minimizer of the auxiliary loss $J_{a,h}$ and that $\hf_{\calA,h}\in\calA_{\tilde R}$ for some constant $\tilde R>0$ independent of $h$. Then, we have
\begin{equation}\label{eq:aux_l1_generalization}
\int_0^T |\hf_{\calA,h}(\Bx(t))-f(\Bx(t))|~\odt\leq C\Big(\kappa_2(\BA_h)(h^p+e_{\calA})+\tilde{R}h^p\Big),
\end{equation}
where $C$ is independent of $h$ and $\calA$, $\kappa_2(\BA_h)$ is the $2$-condition number of $\BA_h$, and $e_{\calA}$ satisfies $e_{\calA}>\inf_{u\in\calA}\sup_{\Bz\in\Omega_{\calT}}|u(\Bz)-f(\Bz)|$.
\end{theorem}
\begin{proof}
By Theorem 5.1 in \cite{Du2022}, the auxiliary grid estimate gives
\begin{equation}\label{eq:thm4.2-01}
|\hf_{\calA,h}-f|_{2,h}\leq \tilde{C}\kappa_2(\BA_h)(h^p+e_{\calA}),
\end{equation}
where $\tilde{C}$ is a constant independent of $h$ and $\calA$. Then, for $\hf_{\calA,h}\in\calA_{\tilde{R}}$, let $\hat{g}_h(t)=\hf_{\calA,h}(\Bx(t))-f(\Bx(t))$ for $0\leq t\leq T$. Using \Cref{lem:sampling_l}, we obtain
\begin{equation}\label{eq:thm4.2-02}
\|\hat{g}_h\|_{L^2(0,T)}\leq\tilde{C}\Big[\Big(h\sum_{n\in\calI_h}|\hat{g}_h(t_n)|^2\Big)^{\frac 12}+\tilde{R}h^p\Big],
\end{equation}
where the constant $\tilde{C}$ is independent of $h$. Since $\hat{g}_h(t_n)=\hf_{\calA,h}(\Bx_n)-f(\Bx_n)$ and $\calI_h\subset\{0,\ldots,N\}$, there exists a constant $C_T>0$ independent of $h$, such that $h|\calI_h|\leq C_T$. Therefore,
\begin{equation}\label{eq:thm4.2-03}
\begin{aligned}
\Big(h\sum_{n\in\calI_h}|\hat{g}_h(t_n)|^2\Big)^{\frac 12}&=(h|\calI_h|)^{\frac 12}\Big(\frac{1}{|\calI_h|}\sum_{n\in\calI_h}|\hat{g}_h(t_n)|^2\Big)^{\frac 12} &\leq C_T|\hf_{\calA,h}-f|_{2,h}.
\end{aligned}
\end{equation}

Combining \eqref{eq:thm4.2-01}-\eqref{eq:thm4.2-03}, we obtain
\begin{equation}\label{eq:thm4.2-04}
\|\hat{g}_h(t)\|_{L^2(0,T)}\leq \hat{C}\Big(\kappa_2(\BA_h)(h^p+e_{\calA})+h^p\tilde R\Big).
\end{equation}
Finally, by the Cauchy--Schwarz inequality,
\begin{equation}\label{eq:thm4.2-05}
\int_0^T |\hat{g}_h(t)|~\odt\leq T^{\frac 12}\|\hat{g}_h(t)\|_{L^2(0,T)}.
\end{equation}
Therefore, combining \eqref{eq:thm4.2-04}-\eqref{eq:thm4.2-05} yields \eqref{eq:aux_l1_generalization}.
\end{proof}

\subsection{Trajectory generalization for the non-auxiliary formulation}\label{sec:noaux_generalization}
We next consider the non-auxiliary formulation. For A-B and A-M schemes, the LMM discovery equations do not uniquely determine all involved values of the learned vector field on the grid. The estimates in Section \ref{sec:kernel} show that this nonuniqueness is localized near the endpoints. We combine the boundary-layer estimates with the sampling inequality to obtain whole-trajectory bounds.

For $p\geq1$ and $\hat R>0$, define
\begin{equation}\label{eq:trajectory_w_class}
\calA_{\hat R}=\left\{u\in\calA:\ g_h\in W^{p,1}(0,T),\ |g_h|_{W^{p,1}(0,T)}\leq \hat R\right\},
\end{equation}
where $g_h$ is defined by \eqref{eq:error}. This set collects functions whose trace errors along the observed trajectory are uniformly bounded in $W^{p,1}(0,T)$.

First, we consider the A-B schemes. The grid error \eqref{eq:ab_kernel_decay} has a one-sided boundary layer error near the initial endpoint. Here, we give the whole-trajectory estimate with a similar boundary layer error.

\begin{theorem}\label{thm:noaux_ab_generalization}
Consider an $M$-step A-B scheme with $M\geq2$ and order $p=M$. Assume that the scheme satisfies the discovery-stability condition \eqref{eq:ab_discovery_stability}. Let $\hat R>0$ be independent of $h$, and let $u_h\in\calA_{\hat R}$ be a zero-residual solution of the non-auxiliary loss, i.e., $J_h(u_h)=0$. Then, we have
\begin{equation}\label{eq:noaux_ab_l1}
\int_0^T |u_h(\Bx(t))-f(\Bx(t))|~ \odt\leq C_{M}\Big(hE_0+h^p+h^p\hat{R}\Big),
\end{equation}
where $E_0=\sum_{j=0}^{M-2}|u_h(\Bx_j)-f(\Bx_j)|$ measures the left boundary layer error, namely the discrepancy between the learned and exact vector-field values at the initial grid points, and $C_M$ is independent of $h$ and $\hat R$.
\end{theorem}
\begin{proof}
Given $h>0$, we can define the componentwise local truncation error by $\tau_{h,n}=h^{-1}\sum_{m=0}^M\alpha_m x(t_{n-m})-\sum_{m=0}^M\beta_m f(\Bx(t_{n-m}))$ for $n=M,\ldots,N$. Let $d_i:=u_h(\Bx(t_i))-f(\Bx(t_i))$ for $i=0,\ldots,N-1$. Since $u_h$ has zero non-auxiliary residual and the A-B scheme satisfies $\beta_0=0$, we have
\begin{equation*}
\sum_{m=1}^M\beta_m d_{n-m}=\tau_{h,n},\qquad n=M,\ldots,N.
\end{equation*}
The order-$p$ consistency gives $|\tau_{h,n}|\le C_\tau h^p$. Hence, by \Cref{lem:forced_ab}, for any fixed $\hat\theta\in(\rho_{\rm AB},1)$, we have
\begin{equation}\label{eq:thm4.3-1}
|d_i|\leq \hat{C}\left(\hat{\theta}^i E_0+\sum_{\ell=M}^{\min\{N,i+1\}}\hat{\theta}^{i+1-\ell}h^p\right)\leq \hat{C}_{\hat{\theta}}\left(\hat{\theta}^i E_0+h^p\right),
\end{equation}
where $i=0,\ldots,N-1$ and $\hat{C}_{\hat{\theta}}$ is independent of $h$. Hence,
\begin{equation}\label{eq:thm4.3-2}
h\sum_{i=0}^{N-1}|d_i|\leq C(hE_0+h\sum_{i=0}^{N-1}h^p)\leq C_T(hE_0+h^p),
\end{equation}
where the constant $C_T$ depends on $T$ but independent of $h$. Applying the $L^1$ sampling estimate \eqref{eq:sampling_l1} to $g_h$ with $\calI_h=\{0,\ldots,N-1\}$ gives
\begin{equation}\label{eq:thm4.3-3}
\int_0^T |g_h(t)|~\odt\leq C\Big(h\sum_{i=0}^{N-1}|d_i|+h^p|g_h|_{W^{p,1}(0,T)}\Big).
\end{equation}
Since $u_h\in\calA_{\hat R}$, we have $|g_h|_{W^{p,1}(0,T)}\leq \hat R$. Combining \eqref{eq:thm4.3-2} and \eqref{eq:thm4.3-3} yields \eqref{eq:noaux_ab_l1}.
\end{proof}

The global estimate above contains the boundary layer contribution $hE_0$, which reflects the fact that the first $M-1$ grid values are not determined by the non-auxiliary A-B residual equations. Thus, if $E_0=O(1)$, the global trajectory estimate is generally restricted to first order, regardless of the formal order $p$ of the A-B scheme. On the other hand, this loss of order is localized near the initial endpoint. By the A-B kernel-localization estimate (\Cref{thm:ab_kernel}), the effect of the boundary error is reduced by a factor $\theta^i$ at the $i$-th grid point. The estimate suggests that the influence of the boundary error should decay exponentially away from the initial point. In particular, on a fixed subinterval $[\delta,T]$ with $\delta>0$, the proof below shows that the accumulated boundary layer contribution is controlled by a term proportional to $h\theta^{\delta/h}E_0$. Since $\theta^{\delta/h}=\exp(\delta\log\theta/h)$ decays exponentially as $h\to0^+$ for $0<\theta<1$, this term can be absorbed into the $O(h^p)$ truncation and sampling terms when $E_0$ remains bounded. This leads to the following interior trajectory estimate.

\begin{theorem}\label{thm:noaux_ab_generalization_1}
Under the assumptions and notation of \Cref{thm:noaux_ab_generalization}, suppose that the boundary layer error $E_0(h)$ remains bounded as $h\to0^+$. Then, for every fixed $\delta\in(0,T)$ and $\theta\in(\rho_{\rm AB},1)$, there exists $\hat{h}>0$ such that, for all $0<h<\hat{h}$,
\begin{equation}\label{eq:thm_0}
\int_\delta^T |u_h(\Bx(t))-f(\Bx(t))|~ \odt\leq C_{M,\theta}(C+\hat{R})h^p,
\end{equation}
where the constants $C_{M,\theta}$ and $C$ are independent of $h$ and $\hat R$.
\end{theorem}
\begin{proof}
Fix $\delta\in(0,T)$ and $\theta\in(\rho_{\rm AB},1)$. Similar to \eqref{eq:thm4.3-1}, we have
\begin{equation}\label{eq:thm4.4-3}
|d_i|\leq  \hat{C}_{\theta}(\theta^i E_0+h^p),\qquad i=0,\ldots,N-1.
\end{equation}
Using \eqref{eq:sampling_l1} on $[\delta,T]$, we obtain
\begin{equation}\label{eq:thm4.4-4}
\int_\delta^T |g_h(t)|~\odt\leq C_0\Big(h\sum_{i:t_i\in[\delta,T]}|d_i|+h^p|g_h|_{W^{p,1}(\delta,T)}\Big).
\end{equation}
For $t_i\in[\delta,T]$, we have $i\geq\delta/h$. Hence
\begin{equation}\label{eq:thm4.4-5}
h\sum_{i:t_i\in[\delta,T]}\theta^iE_0\leq h\sum_{i\geq\delta/h}\theta^iE_0\leq C_\theta h\theta^{\delta/h}E_0.
\end{equation}
Combining \eqref{eq:thm4.4-3}-\eqref{eq:thm4.4-5} yields 
\begin{equation}\label{eq:noaux_ab_interior}
\int_\delta^T |g_h(t)|~\odt\leq \hat{C}\Big(h\theta^{\delta/h}E_0+h^p+h^p\hat{R}\Big),
\end{equation}
where $\hat C$ is independent of $h$, $E_0$, and $\hat R$. By the boundedness assumption of $E_0$, there exist $\tilde{C}>0$ and $h_1>0$ such that $E_0(h)\leq \tilde{C}$ for $0<h<h_1$. Hence, for $0<\theta<1$, we have
\begin{equation}
h\theta^{\delta/h}E_0=E_0 h \exp\Big(\frac{\delta\log\theta}{h}\Big)\leq \tilde{C} h \exp\left(\frac{\delta\log\theta}{h}\right).
\end{equation}
The exponential factor $\exp\left(\frac{\delta\log\theta}{h}\right)$ decays faster than any algebraic power of $h$. In particular, since $\exp\left(\frac{\delta\log\theta}{h}\right)=o(h^{p-1})$ for $h\to0^+$, we have $h\theta^{\delta/h}E_0=o(h^p)$ for $h\to0^+$. Consequently, there exists $\hat{h}\in(0,h_1]$ such that, for all
$0<h<\hat{h}$,
\begin{equation}\label{eq:boundary_layer_absorbed}
h\theta^{\delta/h}E_0\leq \tilde{C} h^p .
\end{equation}
Since $u_h\in\calA_{\hat R}$, we have $|g_h|_{W^{p,1}(0,T)}\leq \hat R$. Substituting this bound and combining \eqref{eq:boundary_layer_absorbed} with \eqref{eq:noaux_ab_interior} yields \eqref{eq:thm_0}.
\end{proof}

Next, we consider the non-auxiliary A-M schemes. Unlike the A-B case, the kernel of the A-M discovery matrix exhibits a two-sided boundary layer structure, so both the initial and terminal endpoint errors enter the global trajectory estimate. The same sampling argument, combined with the A-M kernel-localization estimate, yields the following result.

\begin{theorem}\label{thm:noaux_am_generalization}
Consider an $M$-step A-M scheme with order $p=M+1$. Assume that the scheme satisfies the discovery-stability condition \eqref{eq:am_discovery-stability}. Let $\hat R>0$ be independent of $h$, and let $u_h\in\calA_{\hat R}$ be a zero-residual solution of the non-auxiliary loss, i.e., $J_h(u_h)=0$. Then, we have
\begin{equation}\label{eq:noaux_am_l1}
\int_0^T |u_h(\Bx(t))-f(\Bx(t))|~\odt\leq C_M\Big(h(E_0+E_T)+h^p+h^p\hat R\Big),
\end{equation}
where $E_0=\sum_{j=0}^{M-1}|u_h(\Bx_j)-f(\Bx_j)|$ and $E_T=\sum_{j=N-M+1}^{N}|u_h(\Bx_j)-f(\Bx_j)|$ measure the left and right boundary layer errors, respectively. Here $C_M$ is independent of $h$, $u_h$, $E_0$, $E_T$, and $\hat R$.

Moreover, suppose that the boundary layer errors $E_0$ and $E_T$ remain bounded as $h\to0^+$. Then, for every fixed $\delta\in(0,T/2)$ and $\theta\in(\rho_{\rm AM},1)$, there exists $\hat{h}>0$ such that, for all $0<h<\hat{h}$,
\begin{equation}\label{eq:noaux_am_interior}
\int_\delta^{T-\delta} |u_h(\Bx(t))-f(\Bx(t))|~\odt\leq C_{M,\theta}(C+\hat R)h^p,
\end{equation}
where the constants $C_{M,\theta}$ and $C$ are independent of $h$ and $\hat R$.
\end{theorem}
\begin{proof}
This proof is similar to \Cref{thm:noaux_ab_generalization}, so we will only present the key steps here. The one-step A-M scheme is excluded from this result because it does not satisfy the discovery-stability condition \eqref{eq:am_discovery-stability}. Given $h>0$, we have $\sum_{m=0}^M\beta_m d_{n-m}=\tau_{h,n}$ for $n=M,\ldots,N$, where $\tau_{h,n}$ is the componentwise local truncation error and $d_{n-m}=u_h(\Bx(t_{n-m}))-f(\Bx(t_{n-m}))$.

Let $\hat\theta\in(\rho_{\rm AM},1)$ be fixed. By the forced A-M estimate in \Cref{lem:forced_am} and the hypothesis that the A-M scheme has order $p$, we have
\begin{equation}\label{eq:thm4.4-am-1}
h\sum_{i=0}^{N}|d_i|\leq C_M\Big(h(E_0+E_T)+h^p\Big),
\end{equation}
where the constant $C_M$ is also independent of $h$. Applying the $L^1$ sampling estimate \eqref{eq:sampling_l1} to $g_h$ with $\calI_h=\{0,\ldots,N\}$ gives
\begin{equation}\label{eq:thm4.4-am-2}
\int_0^T |g_h(t)|~\odt\leq C\Big(h\sum_{i=0}^{N}|d_i|+h^p|g_h|_{W^{p,1}(0,T)}\Big)\leq C\Big(h\sum_{i=0}^{N}|d_i|+\hat R h^p\Big).
\end{equation}
Combining \eqref{eq:thm4.4-am-1} and \eqref{eq:thm4.4-am-2} yields \eqref{eq:noaux_am_l1}.

Next, for the interior estimate, fix $\delta\in(0,T/2)$ and $\theta\in(\rho_{\rm AM},1)$. Applying \Cref{lem:forced_am} with this $\theta$ gives
\begin{equation}\label{eq:thm4.4-am-3}
|d_i|\leq C_{M,\theta}\left(\theta^iE_0+\theta^{N-i}E_T+h^p\right),\qquad i=0,\ldots,N.
\end{equation}
Using the $L^1$ sampling estimate on $[\delta,T-\delta]$, we obtain
\begin{equation}\label{eq:thm4.4-am-4}
\int_\delta^{T-\delta} |g_h(t)|~\odt\leq C\Big(h\sum_{i:t_i\in[\delta,T-\delta]}|d_i|+h^p|g_h|_{W^{p,1}(\delta,T-\delta)}\Big).
\end{equation}
For $t_i\in[\delta,T-\delta]$, we have $i\geq \delta/h$ and $N-i\geq \delta/h$. Therefore,
\begin{equation}\label{eq:thm4.4-am-5}
h\sum_{i:t_i\in[\delta,T-\delta]}(\theta^iE_0+\theta^{N-i}E_T)\leq C_\theta h\theta^{\delta/h}(E_0+E_T).
\end{equation}
Combining \eqref{eq:thm4.4-am-3}--\eqref{eq:thm4.4-am-5} with the last two bounds yields \begin{equation}
\int_\delta^{T-\delta} |g_h(t)|~\odt\leq C_{M,\theta}\Big(h\theta^{\delta/h}(E_0+E_T)+h^p+h^p\hat R\Big).
\end{equation}
By the boundedness assumption of $E_0$ and $E_T$, similar to the proof of \Cref{thm:noaux_ab_generalization_1}, there exist $\tilde{C}>0$ and $h_1>0$, such that $h\theta^{\delta/h}(E_0+E_T)\leq \tilde{C}h^p$ for $0<h<h_1$. This proves \eqref{eq:noaux_am_interior}.
\end{proof}

The auxiliary error estimate (\Cref{thm:aux_generalization}) and non-auxiliary error estimate (\Cref{thm:noaux_ab_generalization}-\Cref{thm:noaux_am_generalization}) differ in the source of their leading error terms. In the auxiliary formulation, the auxiliary conditions remove the kernel space of the discovery matrix, so the whole-trajectory estimate is obtained directly from a grid-level error bound and a sampling inequality. If the condition number $\kappa_2(\BA_h)$ is uniformly bounded, the approximation error satisfies $e_{\calA}=O(h^p)$, and the trajectory regularity bound is uniform in $h$, then the auxiliary estimate gives an $O(h^p)$ whole-trajectory error. In the non-auxiliary formulation, the residual equations leave boundary degrees of freedom undetermined. Kernel localization confines these degrees of freedom to endpoint layers, but their contribution still appears in the global trajectory error.

\section{Numerical experiments}\label{sec:numerics}

This section presents a noise-free single-trajectory experiment. The purpose is to compare the observed trajectory error with the convergence behavior predicted by the auxiliary and non-auxiliary trajectory estimates in \Cref{sec:generalization}. We report trajectory errors for both formulations and for several A-B and A-M schemes. The zero-residual assumption in the non-auxiliary estimates is used to isolate the discretization and boundary-layer errors; in the neural-network experiments, the residual is minimized only approximately, and the remaining optimization error is treated as an additional numerical error.

\subsection{Error metrics}
Let $\hat{\Bf}_h=[\hat f_{h,1}~\ldots~\hat f_{h,d}]^\top\in\mathbb{R}^d$ be the network approximating the original governing function $\Bf$. On a single trajectory, we define the trajectory error
\begin{equation}\label{eq:num_traj_error_single}
e_{{\rm traj}}=\frac{1}{d}\sum_{j=1}^d\frac{\int_0^T\left|\hat f_{h,j}(\Bx(t))-f_j(\Bx(t))\right|~\odt}{\int_0^T\left|f_j(\Bx(t))\right|~\odt},
\end{equation}
where the integrals are approximated by the composite trapezoidal rule.

\subsection{Experimental settings}

The common experimental settings are summarized as follows.
\begin{itemize}
\item {\em Environment.}
The experiments are performed in the Python 3.12.3 environment. We use PyTorch for neural network implementation.

\item {\em Initialization of variables.}
The network parameters $\{\BW_l, \Bb_l\}$ are randomly initialized with uniform distribution by
\begin{equation*}
\BW_l, \Bb_l\sim U(-W^{-1/2},W^{-1/2}).
\end{equation*}

\item {\em Optimizer and hyperparameters.}
The optimization is performed by Adam for $3\times 10^4$ epochs. Since each scalar problem contains at most $N+1$ residual and auxiliary constraints, all constraints are used in each optimization step. Therefore, the training is effectively full-batch. The learning rate is decreased from $10^{-2}$ to $10^{-5}$ during training.

\item {\em Randomness.}
To reduce the effect of initialization randomness, each configuration is trained with multiple random seeds, which correspond to different initial guesses for the optimization. In the convergence plots below, we report the mean errors over the available seed runs for each method, grid size, and formulation.
\end{itemize}

\subsection{Noise-free trajectory data}\label{sec:numerics_case1}
We consider the model problem
\begin{equation*}
\dot x_1=x_2,\quad
\dot x_2=-x_1,\quad
\dot x_3=1/x_2^2,\quad t\in[0,T],
\end{equation*}
with initial condition $[x_1,x_2,x_3]_{t=0}=[0,1,0]$ and $T=1$. The exact solution is $x(t)=(\sin t,\cos t,\tan t)$. We assign the time step size $h=1/N$ with $N=4,6,8,10,12,14,16$. The governing function is approximated by an FNN with depth $L=5$, width $W=640$, and scaled Softplus activation. For each component of the vector field, we train an independent scalar network $u_{\theta,j}:\mathbb R^3\to\mathbb R$ for $j=1,2,3$. We test the A-B and A-M formulations ($M=1,2,3$) with auxiliary conditions and without auxiliary conditions. The auxiliary conditions are imposed by one-sided finite-difference approximations near the initial time. The training data consist of a single noise-free trajectory sampled at $\Bx_i=\Bx(t_i)$ with $t_i=ih$, $i=0,\ldots,N$. For each scalar component, the non-auxiliary losses contain $N-M+1$ LMM residual points. With auxiliary conditions, the corresponding numbers become $N$ for A-B schemes and $N+1$ for A-M schemes. 

\begin{figure}[htbp!]
\centering
\includegraphics[width=0.78\textwidth]{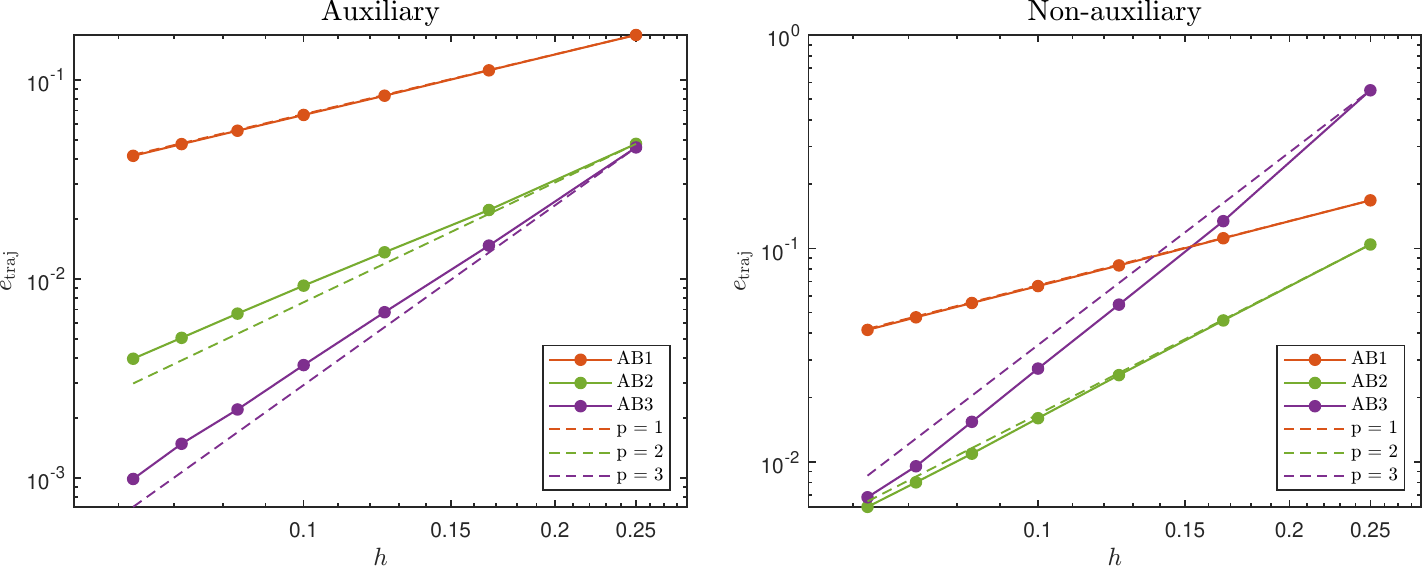}
\caption{Mean trajectory error $e_{\rm traj}$ over random seeds for A-B schemes. Dashed lines are scaled reference slopes with rates $p=1,2,3$.}
\label{fig:convergence_ab}
\end{figure}

\Cref{fig:convergence_ab} shows the mean trajectory error $e_{\rm traj}$ for auxiliary and non-auxiliary A-B schemes. For $M=1$, the two formulations coincide, consistent with the fact that the one-step A-B discovery matrix has a trivial kernel and hence no additional boundary degrees of freedom. For $M=2$ and $M=3$, both formulations show decreasing trajectory errors as $h$ decreases, with observed slopes broadly consistent with the corresponding reference orders. The non-auxiliary results are close to the second- and third-order reference slopes for $M=2$ and $M=3$, respectively. Although \Cref{thm:noaux_ab_generalization} contains the additional boundary-layer contribution $hE_0$, the observed behavior suggests that this term is not the leading contribution in this noise-free experiment.

\begin{figure}[htbp!]
\centering
\includegraphics[width=0.78\textwidth]{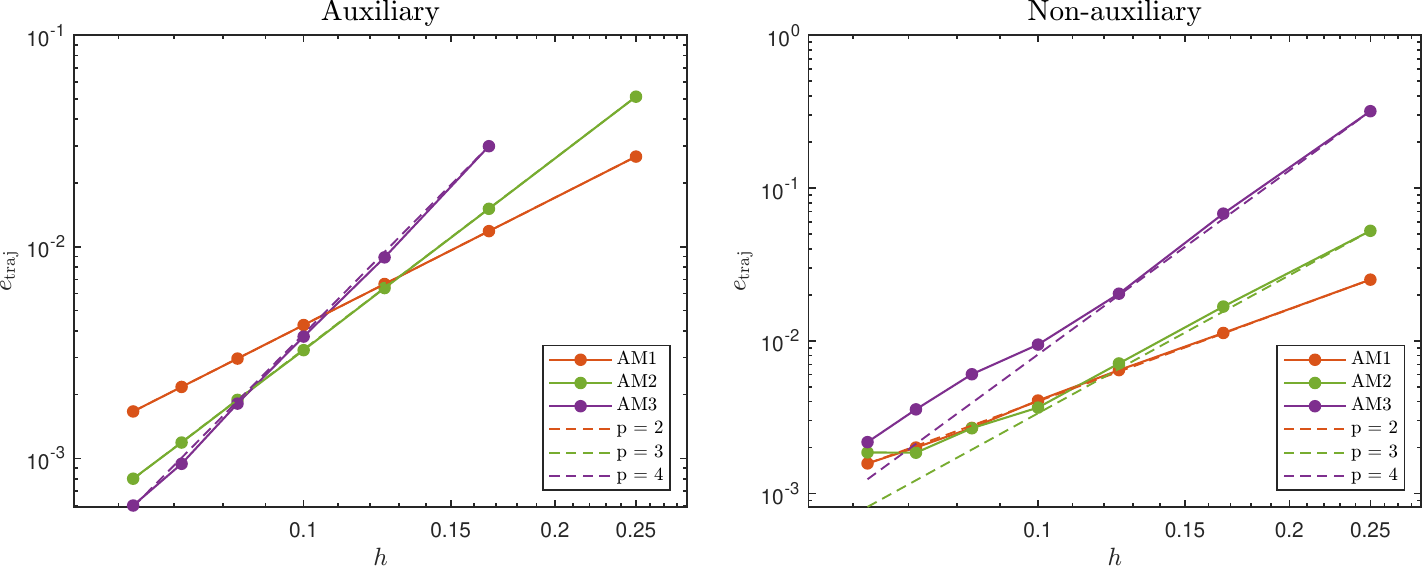}
\caption{Mean trajectory error $e_{\rm traj}$ over random seeds for A-M schemes. Dashed lines are scaled reference slopes with rates $p=2,3,4$.}
\label{fig:convergence_am}
\end{figure}

\Cref{fig:convergence_am} shows the mean trajectory error $e_{\rm traj}$ for A-M schemes. The one-step A-M scheme is not covered by the non-auxiliary A-M theory because its discovery polynomial has a unit-modulus root, but it is included as a numerical reference. For $M=1$, both formulations exhibit approximately second-order convergence. For $M=2$ and $M=3$, the trajectory errors decrease over the available mesh sizes and follow the corresponding reference slopes. The auxiliary A-M result with $M=3$ and $h=1/4$ is omitted because the one-sided auxiliary finite-difference conditions require trajectory samples beyond those available on this grid. These observations are consistent with the theoretical distinction between the two formulations. The auxiliary conditions remove the kernel degrees of freedom, whereas the non-auxiliary estimate in \Cref{thm:noaux_am_generalization} contains the endpoint-layer contribution $h(E_0+E_T)$. If this contribution were $O(h)$ and dominant, the global trajectory error could be limited to first order. The observed higher-order decay suggests that the endpoint-layer term is not dominant in the present experiment, although both approximation and residual-minimization errors still affect the finite-mesh behavior.

\section{Conclusion}\label{sec:conclusion}
This paper studied two mechanisms in LMM-based discovery of dynamical systems that are not captured by existing grid-level convergence theory: the boundary-layer localization of nonuniqueness in non-auxiliary discovery systems and the conversion of grid-level estimates into error bounds along the observed trajectory. For non-auxiliary A-B and A-M formulations, the corresponding discovery systems are underdetermined, but their kernels are highly structured. Under the corresponding discovery-stability conditions, the differences between zero-residual grid solutions of the A-B scheme are exponentially localized near the initial indices, while the corresponding differences for A-M schemes exhibit two-sided boundary layers near the initial and terminal indices. These results indicate that the localized nonuniqueness of the non-auxiliary formulation originates from the boundary layer structure of the LMM algebraic system itself.

We also derived whole-trajectory error estimates along a fixed observed trajectory. After restriction to the trajectory, the vector-field error becomes a one-dimensional trace error in time, allowing deterministic sampling inequalities to convert grid estimates into continuous-time bounds. For the auxiliary formulation, this yields an $O(h^p)$ trajectory estimate under the stated approximation, conditioning, and trace-regularity assumptions. For non-auxiliary A-B and A-M formulations, the global estimates contain additional boundary-layer terms, respectively $O(hE_0)$ and $O(h(E_0+E_T))$. These terms explain the possible loss of global order, while their exponential damping on fixed interior subintervals recovers the $O(h^p)$ behavior when the boundary-layer errors remain bounded. The numerical experiments support these conclusions.

Future work includes establishing a connection between the trajectory regularity assumptions used in our analysis and verifiable properties of trained neural networks, such as norm control, spectral bias, or explicit regularization. Another direction is to extend the current deterministic estimates to noisy trajectory data and cases of model mismatch. These extensions will further elucidate the stability and generalizability of LMM-based discovery methods in practical learning environments.

\bibliographystyle{siamplain}
\bibliography{references}

\end{document}